\documentclass[11pt]{amsart}
\usepackage{amsmath,amssymb}

\begin{document}

\newtheorem{thm}{Theorem}[section]
\newtheorem{lem}[thm]{Lemma}
\newtheorem{prop}[thm]{Proposition}
\newtheorem{cor}[thm]{Corollary}
\newtheorem{defi}[thm]{Definition}
\newtheorem{remark}[thm]{Remark}
\newtheorem{conj}[thm]{Conjecture}
\newtheorem{problem}[thm]{Problem}

\numberwithin{equation}{section}

\newcommand{\Z}{{\mathbb Z}} 
\newcommand{\Q}{{\mathbb Q}}
\newcommand{\R}{{\mathbb R}}
\newcommand{\C}{{\mathbb C}}
\newcommand{\N}{{\mathbb N}}
\newcommand{\FF}{{\mathbb F}}
\newcommand{\fq}{\mathbb{F}_q}

\newcommand{\rmk}[1]{\footnote{{\bf Comment:} #1}}

\renewcommand{\mod}{\;\operatorname{mod}}
\newcommand{\ord}{\operatorname{ord}}
\newcommand{\TT}{\mathbb{T}}
\renewcommand{\i}{{\mathrm{i}}}
\renewcommand{\d}{{\mathrm{d}}}
\renewcommand{\^}{\widehat}
\newcommand{\HH}{\mathbb H}
\newcommand{\Vol}{\operatorname{vol}}
\newcommand{\area}{\operatorname{area}}
\newcommand{\tr}{\operatorname{tr}}
\newcommand{\norm}{\mathcal N} 
\newcommand{\intinf}{\int_{-\infty}^\infty}
\newcommand{\ave}[1]{\left\langle#1\right\rangle} 
\newcommand{\Var}{\operatorname{Var}}
\newcommand{\Prob}{\operatorname{Prob}}
\newcommand{\sym}{\operatorname{Sym}}
\newcommand{\disc}{\operatorname{disc}}
\newcommand{\CA}{{\mathcal C}_A}
\newcommand{\cond}{\operatorname{cond}} 
\newcommand{\lcm}{\operatorname{lcm}}
\newcommand{\Kl}{\operatorname{Kl}} 
\newcommand{\leg}[2]{\left( \frac{#1}{#2} \right)}  
\newcommand{\Li}{\operatorname{Li}}

\newcommand{\sumstar}{\sideset \and^{*} \to \sum}

\newcommand{\LL}{\mathcal L} 
\newcommand{\sumf}{\sum^\flat}
\newcommand{\Hgev}{\mathcal H_{2g+2,q}}
\newcommand{\USp}{\operatorname{USp}}
\newcommand{\conv}{*}
\newcommand{\dist} {\operatorname{dist}}
\newcommand{\CF}{c_0} 
\newcommand{\kerp}{\mathcal K}

\title[Mean Value Theorems for $L$--functions over Prime Polynomials]{Mean Value Theorems for $L$--functions over Prime Polynomials for the Rational Function Field}

\author{Julio C. Andrade and Jonathan P. Keating}
\address{Institute for Computational and Experimental Research in Mathematics (ICERM), Brown University, 121 South Main Street, Providence, RI, 02903, USA}
\email{julio\_andrade@brown.edu}

\address{School of Mathematics, University of Bristol, Bristol BS8 1TW, UK}
\email{j.p.keating@bristol.ac.uk}

\thanks{JCA is supported by a NSF Postdoctoral Grant and an ICERM--Brown University Postdoctoral Research Fellowship. JPK is sponsored by the Leverhulme Trust and the Air Force Office of Scientific Research, Air Force Material Command, USAF, under grant number FA8655-10-1-3088. The U.S. Government is authorized to reproduce and distribute reprints for Governmental purpose notwithstanding any copyright notation thereon.}
\subjclass[2010]{11G20(Primary), 11M38, 11M50, 14G10(Secondary)}
\keywords{finite fields, function fields, hyperelliptic curves, moments of quadratic Dirichlet $L$--functions, prime polynomials}


\begin{abstract}
The first and second moments are established for the family of quadratic Dirichlet $L$--functions over the rational function field at the central point $s=\tfrac{1}{2}$ where the character $\chi$ is defined by the Legendre symbol for polynomials over finite fields and runs over all monic irreducible polynomials $P$ of a given odd degree. Asymptotic formulae are derived for fixed finite fields when the degree of $P$ is large. The first moment obtained here is the function field analogue of a result due to Jutila in the number--field setting. The approach is based on classical analytical methods and relies on the use of the analogue of the approximate functional equation for these $L$--functions. 
\end{abstract}

\maketitle

\section{Introduction}

It is a much studied problem in analytic number theory to obtain asymptotic formulae for the moments of families $L$--functions. For the family of quadratic Dirichlet $L$--functions $L(s,\chi_{d})$, where $\chi_{d}$ is a real primitive Dirichlet character modulo $d$ defined by the Jacobi symbol $\chi_{d}(n)=\left(\frac{d}{n}\right)$, the problem is to establish asymptotics for
\begin{equation}
\label{1.1}
\sideset{}{^{*}}\sum_{d\leq X}L(\tfrac{1}{2},\chi_{d})^{k},
\end{equation}
in the limit as $X\rightarrow\infty$ and where the sum includes fundamental discriminants $d$. For $k=1,2$, Jutila \cite{J} established the following asymptotic formulae
\begin{equation}
\sideset{}{^{*}}\sum_{d\leq X}L(\tfrac{1}{2},\chi_{d})\sim c_{1}X\log X,
\end{equation} 
and
\begin{equation}
\sideset{}{^{*}}\sum_{d\leq X}L(\tfrac{1}{2},\chi_{d})^{2}\sim c_{2}X(\log X)^{3},
\end{equation}
where $c_{1}$ and $c_{2}$ are computable constants given in terms of Euler products and factors involving the Riemann zeta function. For $k=3$, Soundararajan \cite{S} proved that
\begin{equation}
\sideset{}{^{*}}\sum_{d\leq X}L(\tfrac{1}{2},\chi_{8d})^{3}\sim c_{3}X(\log X)^{6},
\end{equation}
where $d$ is an odd, square--free and positive number, so that $\chi_{8d}$ is a real, even primitive Dirichlet character with conductor $8d$ and $c_{3}$ is a constant. Recently, Soundararajan and Young \cite{S-Y} have claimed that under the Generalized Riemann Hypothesis they are able to establish an asymptotic formula for the fourth power moment for this family of $L$--functions, i.e.,
\begin{equation}
\sideset{}{^{*}}\sum_{d\leq X}L(\tfrac{1}{2},\chi_{8d})^{4}\sim c_{4}X(\log X)^{10},
\end{equation}
where $c_{4}$ is a computable constant. No other asymptotic values are known for the mean values of quadratic Dirichlet $L$--functions at the centre of the critical strip.

Using results from Random Matrix Theory, Keating and Snaith \cite{KeS2} have put forward a conjecture for the leading order asymptotic for all moments of quadratic Dirichlet $L$--function which agrees with the results listed above.
\begin{conj}[Keating--Snaith]
For $k$ fixed with $\mathfrak{R}(k)\geq0$, as $X\rightarrow\infty$
\begin{equation}
\frac{1}{X^{*}}\sideset{}{^{*}}\sum_{0<d\leq X}L(\tfrac{1}{2},\chi_{8d})^{k}\sim a_{k,Sp}\frac{G(k+1)\sqrt{\Gamma(k+1)}}{\sqrt{G(2k+1)\Gamma(2k+1)}}(\log X)^{k(k+1)/2}
\end{equation}
where
$$a_{k,Sp}=2^{-k(k+2)/2}\prod_{p\geq3}\frac{(1-\frac{1}{p})^{k(k+1)/2}}{1+\frac{1}{p}}\left(\frac{(1-\frac{1}{\sqrt{p}})^{-k}+(1+\frac{1}{\sqrt{p}})^{-k}}{2}+\frac{1}{p}\right)$$
and $G(z)$ is Barnes' $G$--function.
\end{conj}
Conjectures for the lower order terms are presented in \cite{CFKRS} and \cite{DGH}.

A similar problem involving moments of quadratic Dirichlet $L$--functions was considered by Goldfeld and Viola \cite{GV}, who have conjectured an asymptotic formula for
\begin{equation}
\sum_{\substack{p\leq X \\ p\equiv3(\mod4)}}L(\tfrac{1}{2},\chi_{p}),
\end{equation}
where $\chi_{p}(n)=\left(\frac{n}{p}\right)$ is defined by the Legendre symbol. In this context Jutila \cite{J} established the following asymptotic formula
\begin{equation}
\label{jutilaprime}
\sum_{\substack{p\leq X \\ p\equiv3(\mod4)}}(\log p)L(\tfrac{1}{2},\chi_{p})=\frac{1}{4}X\log X+O(X(\log X)^{\varepsilon}).
\end{equation}  
It is natural to ask about higher moments for the family of quadratic Dirichlet $L$--functions associated to $\chi_{p}$. This problem has the same flavour as that involving the mean values of quadratic Dirichlet $L$--functions over fundamental discriminants and we formulate it as follows:
\begin{problem}
\label{prob1}
Establish asymptotic formulas for
\begin{equation}
\label{1.9}
\sum_{\substack{p\leq X \\ p\equiv3(\mod4)}}L(\tfrac{1}{2},\chi_{p})^{k},
\end{equation} 
when $X\rightarrow\infty$ and $k>1$.
\end{problem}

In this paper we study the function field analogue of this problem in the same spirit as the recent result obtained in \cite{AK} for the first moment of quadratic Dirichlet $L$--functions over the rational function field $\mathbb{F}_{q}(T)$. Our aim is to obtain asymptotic formulae for the first and second moments for the function field analogue of Problem \ref{prob1} as developed in the next section. Higher moments are studied in \cite{And1} 

\section{Statement of Results}
Before stating our main results we establish some notation and some preliminary facts about quadratic Dirichlet $L$--functions for function fields.

\subsection{Zeta function of Curves}

We start with $\mathbb{F}_{q}$ denoting a finite field of odd cardinality, $A=\mathbb{F}_{q}[T]$ polynomials in the variable $T$ with coefficients in $\mathbb{F}_{q}$, and $k=\mathbb{F}_{q}(T)$ the rational function field over $\mathbb{F}_{q}$. Let $C$ be any smooth, projective, geometrically connected curve of genus $g\geq1$ defined over the finite field $\mathbb{F}_{q}$. Artin \cite{A} defined the zeta function of the curve $C$ as
\begin{equation}
Z_{C}(u):=\exp\left(\sum_{n=1}^{\infty}N_{n}(C)\frac{u^{n}}{n}\right), \ \ \ \ \ |u|<1/q
\end{equation}
with $N_{n}(C):=\mathrm{Card}(C(\mathbb{F}_{q}))$ the number of points on $C$ where the coordinates are in a field extension $\mathbb{F}_{q^{n}}$ of $\mathbb{F}_{q}$ of degree $n\geq1$. It turns out that, as shown by Weil \cite{W}, the zeta function associated to $C$ is a rational function of the form
\begin{equation}\label{eq:zetaC}
Z_{C}(u)=\frac{L_{C}(u)}{(1-u)(1-qu)},
\end{equation}
where $L_{C}(u)\in\mathbb{Z}[u]$ is a polynomial of degree $2g$ that satisfies the functional equation
\begin{equation}\label{eq:funceq}
L_{C}(u)=(qu^{2})^{g}L_{C}\left(\frac{1}{qu}\right).
\end{equation}
The Riemann Hypothesis for curves over finite fields, established by Weil \cite{W}, asserts that the zeros of $L_{C}(u)$ all lie on the circle $|u|=q^{-1/2}$, i.e.,
\begin{equation}
L_{C}(u)=\prod_{j=1}^{2g}(1-\alpha_{j}u), \ \ \ \ \ \mathrm{with} \ \ |\alpha_{j}|=\sqrt{q} \ \ \mathrm{for \ all}\ j.
\end{equation}

\subsection{Essential Facts about $\mathbb{F}_{q}[T]$}
In this paper we denote the \textit{norm} of a polynomial $f\in A$ by $|f|:=q^{\mathrm{deg}(f)}$ for $f\neq0$ and $|f|=0$ for $f=0$, and we call a monic irreducible polynomial $P\in A$ a \textit{prime polynomial}.

The zeta function of $A=\mathbb{F}_{q}[T]$ will be denoted by $\zeta_{A}(s)$ and is defined in the following natural way
\begin{equation}
\label{eq:zetaA}
\zeta_{A}(s):=\sum_{\substack{f\in A \\ f \ \mathrm{monic}}}\frac{1}{|f|^{s}}=\prod_{\substack{P \ \mathrm{monic} \\ \mathrm{irreducible}}}\left(1-|P|^{-s}\right)^{-1}, \ \ \ \ \ \ \mathfrak{R}(s)>1.
\end{equation}
In this case the zeta function $\zeta_{A}(s)$ is simply given by
\begin{equation}\label{eq:zetaA1}
\zeta_{A}(s)=\frac{1}{1-q^{1-s}}.
\end{equation}
The fact that this has a simple pole and no zeros leads to the analogue of the Prime Number Theorem for polynomials in $A=\mathbb{F}_{q}[T]$
\begin{thm}[Prime Polynomial Theorem]
\label{thm:pnt}
If $\pi_{A}(n)$ denotes the number of monic irreducible polynomials in $A$ of degree $n$, then, 
\begin{equation}
\pi_{A}(n)=\frac{q^{n}}{n}+O\left(\frac{q^{\tfrac{n}{2}}}{n}\right).
\end{equation}
\end{thm}

\subsection{Quadratic Dirichlet $L$--function for $\chi_{P}$}

Let $P\in A$ be a monic irreducible polynomial.  We denote by $\chi_{P}$ the quadratic character defined in terms of the quadratic residue symbol for $\mathbb{F}_{q}[T]$
\begin{equation}
\chi_{P}(f)=\left(\frac{P}{f}\right),
\end{equation} 
where $f\in A$. For more details see \cite[Chapters 3, 4]{Ro}. We will make use of the quadratic reciprocity law for polynomials in $A$ 
\begin{thm}[Quadratic reciprocity]
\label{reciprocity}
Let $A,B\in\mathbb{F}_{q}[T]$ be relatively prime and $A\neq0$ and $B\neq0$. Then,
\begin{equation}
\left(\frac{A}{B}\right)=\left(\frac{B}{A}\right)(-1)^{((q-1)/2)\mathrm{deg}(A)\mathrm{deg}(B)}=\left(\frac{B}{A}\right)(-1)^{((|A|-1)/2)((|B|-1)/2)}.
\end{equation}
\end{thm}

The $L$--function attached to the character $\chi_{P}$ is defined by
\begin{equation}
L(s,\chi_{P}):=\sum_{\substack{f\in A \\ f \ \mathrm{monic}}}\frac{\chi_{P}(f)}{|f|^{s}}=\prod_{\substack{Q \ \mathrm{monic} \\ \mathrm{irreducible}}}\left(1-\frac{\chi_{P}(Q)}{|Q|^{s}}\right)^{-1}, \ \ \ \mathfrak{R}(s)>1.
\end{equation}

Henceforth we consider $P$ to be a monic irreducible polynomial such that $\mathrm{deg}(P)$ is odd and $q\equiv1(\bmod \ 4)$. Then \cite[Propositions 4.3, 14.6 and 17.7]{Ro} $L(s,\chi_{P})$ is a polynomial in $u=q^{-s}$ of degree $\mathrm{deg}(P)-1$ and
\begin{equation}
L(s,\chi_{P})=\mathcal{L}(u,\chi_{P})=L_{C_{P}}(u),
\end{equation}
where $L_{C_{P}}(u)$ is the numerator of the zeta function associated to the hyperelliptic curve given in affine form by
\begin{equation}
C_{P}:y^{2}=P(T)
\end{equation}
with
\begin{equation}
P(T)=T^{2g+1}+a_{2g}T^{2g}+\cdots+a_{1}T+a_{0}
\end{equation}
a monic irreducible polynomial in $A$ of degree $2g+1$. 

The following proposition is quoted from Rudnick \cite{Ru} and the main ingredient to establish it is the Riemann Hypothesis for curves

\begin{prop}
\label{bound}
If we assume $f\in A$ is monic, $\mathrm{deg}(f)>0$ and $f$ is not a perfect square then we have
\begin{equation}
\left|\sum_{\substack{P \ \mathrm{prime} \\ \mathrm{deg}(P)=n}}\left(\frac{f}{P}\right)\right|\ll\frac{\mathrm{deg}(f)}{n}q^{n/2}.
\end{equation}
\end{prop}
\vspace{0.5cm}

\subsection{The main results}
We now present the main results of this paper.

\begin{thm}
\label{thm1}
Let $\mathbb{F}_{q}$ be a fixed finite field of odd cardinality with $q\equiv1(\bmod \ 4)$. Then for every $\varepsilon>0$ we have,
\begin{equation}
\sum_{\substack{P \ \mathrm{monic} \\ \mathrm{irreducible} \\ \mathrm{deg}(P)=2g+1}}(\log_{q}|P|)L(\tfrac{1}{2},\chi_{P})=\frac{|P|}{2}(\log_{q}|P|+1)+O(|P|^{\tfrac{3}{4}+\varepsilon}).
\end{equation}
\end{thm}

This theorem also appears as part of the Ph.D thesis \cite{And} of the first author. This is the exact function field analogue of Jutila's result \eqref{jutilaprime} for number--fields. Note that the function field theorem above has a saving in the error term when compared with the number--field result \eqref{jutilaprime}. 

\begin{thm}
\label{thm2}
Using the same notation as before, for a fixed finite field $\mathbb{F}_{q}$ we have
\begin{multline}
\sum_{\substack{P \ \mathrm{monic} \\ \mathrm{irreducible} \\ \mathrm{deg}(P)=2g+1}}L(\tfrac{1}{2},\chi_{P})^{2}=\frac{1}{24}\frac{1}{\zeta_{A}(2)}|P|(\log_{q}|P|)^{2}+O(|P|(\log_{q}|P|)).\ \ \ \ \ \ \ \ \ \ \ \ \ \ \ \ \ \ \ \ \ \ \ \ \ \ \ \ \ \ \ \ \ \ 
\end{multline}
\end{thm}

We have the following Corollary

\begin{cor}
\begin{equation}
\sum_{\substack{P \ \mathrm{monic} \\ \mathrm{irreducible} \\ \mathrm{deg}(P)=2g+1 \\ L(\tfrac{1}{2},\chi_{P})\neq0}}1\gg\frac{|P|}{(\log_{q}|P|)^{2}}.
\end{equation}
\end{cor}
\begin{proof}
From Theorems \ref{thm1} and \ref{thm2} we have
\begin{equation}
\label{cor1}
\sum_{\substack{P \ \mathrm{monic} \\ \mathrm{irreducible} \\ \mathrm{deg}(P)=2g+1}}L(\tfrac{1}{2},\chi_{P})\sim k_{1}|P|
\end{equation}
and
\begin{equation}
\label{cor2}
\sum_{\substack{P \ \mathrm{monic} \\ \mathrm{irreducible} \\ \mathrm{deg}(P)=2g+1}}L(\tfrac{1}{2},\chi_{P})^{2}\sim k_{2}|P|(\log_{q}|P|)^{2},
\end{equation}
where $k_{1}$ and $k_{2}$ are the constants given in the above theorems. By Cauchy-–Schwarz inequality follows that the number of monic irreducible polynomials $P$ with $\mathrm{deg}(P)=2g+1$ such that $L(\tfrac{1}{2},\chi_{P})\neq0$ exceeds the ratio of the square of the quantity in \eqref{cor1} to the quantity in \eqref{cor2}.
\end{proof}

\section{The First Moment}

Setting $D=P$ in Lemma 3.3 from \cite{AK}, we may write $L(\tfrac{1}{2},\chi_{P})$ as

\begin{equation}
\label{eq3.1}
L(\tfrac{1}{2},\chi_{P})=\sum_{n=0}^{g}\sum_{\substack{f_{1} \ \mathrm{monic} \\ \mathrm{deg}(f_{1})=n}}\chi_{P}(f_{1})q^{-\tfrac{n}{2}}+\sum_{m=0}^{g-1}\sum_{\substack{f_{2} \ \mathrm{monic} \\ \mathrm{deg}(f_{2})=m}}\chi_{P}(f_{2})q^{-\tfrac{m}{2}}.
\end{equation}

We need to average both double sums in the right--hand side of \eqref{eq3.1} over monic irreducible polynomials of degree $2g+1$. However they are clearly related and we will only need to calculate one of them to obtain the result for the other. Therefore we will focus on the average of the first double sum in \eqref{eq3.1}. We can write this as
\begin{multline}
\label{eq3.2}
\sum_{n=0}^{g}\sum_{\substack{f_{1} \ \mathrm{monic} \\ \mathrm{deg}(f_{1})=n}}\chi_{P}(f_{1})q^{-\tfrac{n}{2}}\\
=\sum_{n=0}^{g}\sum_{\substack{f_{1} \ \mathrm{monic} \\ \mathrm{deg}(f_{1})=n \\ f_{1}=\square}}\chi_{P}(f_{1})q^{-\tfrac{n}{2}}+\sum_{n=0}^{g}\sum_{\substack{f_{1} \ \mathrm{monic} \\ \mathrm{deg}(f_{1})=n \\ f_{1}\neq\square}}\chi_{P}(f_{1})q^{-\tfrac{n}{2}}.
\end{multline}

\subsection{Square Contributions --The Main Term}
In this section we focus our attention on the average of the first double sum in the right hand side of \eqref{eq3.2}. The main result  is 

\begin{prop}
\label{prop3.1}
We have that,
\begin{equation}
\sum_{\substack{P \ \mathrm{monic} \\ \mathrm{irreducible} \\ \mathrm{deg}(P)=2g+1}}\sum_{n=0}^{g}\sum_{\substack{f_{1} \ \mathrm{monic} \\ \mathrm{deg}(f_{1})=n \\ f_{1}=\square}}\chi_{P}(f_{1})q^{-\tfrac{n}{2}}=\frac{|P|}{\log_{q}|P|}\left(\left[\frac{g}{2}\right]+1\right)+O\left(\frac{\sqrt{|P|}}{\log_{q}|P|}g\right),\nonumber
\end{equation}
where $[x]$ denotes the integer part of $x$.
\end{prop}

\begin{proof}
We have,
\begin{eqnarray}
& &\sum_{\substack{P \ \mathrm{monic} \\ \mathrm{irreducible} \\ \mathrm{deg}(P)=2g+1}}\sum_{n=0}^{g}\sum_{\substack{f_{1} \ \mathrm{monic} \\ \mathrm{deg}(f_{1})=n \\ f_{1}=\square}}\chi_{P}(f_{1})q^{-\tfrac{n}{2}}\nonumber\\
&=&\sum_{\substack{n=0 \\ 2\mid n}}^{g}q^{-\tfrac{n}{2}}\sum_{\substack{l \ \mathrm{monic} \\ \mathrm{deg}(l)=\tfrac{n}{2}}}\sum_{\substack{P \ \mathrm{monic} \\ \mathrm{irreducible} \\ \mathrm{deg}(P)=2g+1}}\chi_{P}(l^{2})=\sum_{\substack{n=0 \\ 2\mid n}}^{g}q^{-\tfrac{n}{2}}\sum_{\substack{l \ \mathrm{monic} \\ \mathrm{deg}(l)=\tfrac{n}{2}}}\sum_{\substack{P \ \mathrm{monic} \\ \mathrm{irreducible} \\ \mathrm{deg}(P)=2g+1 \\ (P,l)=1}}1\nonumber\\
&=&\sum_{\substack{n=0 \\ 2\mid n}}^{g}q^{-\tfrac{n}{2}}\sum_{\substack{l \ \mathrm{monic} \\ \mathrm{deg}(l)=\tfrac{n}{2}}}\sum_{\substack{P \ \mathrm{monic} \\ \mathrm{irreducible} \\ \mathrm{deg}(P)=2g+1}}1,\nonumber
\end{eqnarray}
where we obtain the last line from the fact that $\mathrm{deg}(P)=2g+1>\mathrm{deg}(l)$. Making use of the Prime Polynomial Theorem \ref{thm:pnt} we can write\begin{multline}
\sum_{\substack{P \ \mathrm{monic} \\ \mathrm{irreducible} \\ \mathrm{deg}(P)=2g+1}}\sum_{n=0}^{g}\sum_{\substack{f_{1} \ \mathrm{monic} \\ \mathrm{deg}(f_{1})=n \\ f_{1}=\square}}\chi_{P}(f_{1})q^{-\tfrac{n}{2}}\\
=\sum_{\substack{n=0 \\ 2\mid n}}^{g}q^{-\tfrac{n}{2}}\sum_{\substack{l \ \mathrm{monic} \\ \mathrm{deg}(l)=\tfrac{n}{2}}}\left(\frac{q^{2g+1}}{2g+1}+O\left(\frac{q^{g}}{2g+1}\right)\right)\nonumber\\ \\
=\frac{q^{2g+1}}{2g+1}\sum_{m=0}^{\left[\tfrac{g}{2}\right]}1+O\left(\frac{q^{g}}{2g+1}\sum_{m=0}^{\left[\tfrac{g}{2}\right]}1\right)\ \ \ \ \ \ \ \nonumber\\ \\
=\frac{|P|}{\log_{q}|P|}\left(\left[\frac{g}{2}\right]+1\right)+O\left(\frac{\sqrt{|P|}}{\log_{q}|P|}g\right).\ \ \ \ \ \ \ \ \ \ \ \ \ \ \ \ \ \ \ \ \ \nonumber
\end{multline}
\end{proof}

In an analogous way we can prove that

\begin{prop}
\label{prop3.2}
\begin{multline}
\sum_{\substack{P \ \mathrm{monic} \\ \mathrm{irreducible} \\ \mathrm{deg}(P)=2g+1}}\sum_{m=0}^{g-1}\sum_{\substack{f_{2} \ \mathrm{monic} \\ \mathrm{deg}(f_{2})=m \\ f_{2}=\square}}\chi_{P}(f_{2})q^{-\tfrac{m}{2}}\nonumber\\
=\frac{|P|}{\log_{q}|P|}\left(\left[\frac{g-1}{2}\right]+1\right)+O\left(\frac{\sqrt{|P|}}{\log_{q}|P|}g\right).
\end{multline}
\end{prop}

\subsection{Contributions of non--squares}

In this section we prove the following result
\begin{prop}
\label{prop3.3}
\begin{equation}
\sum_{\substack{P \ \mathrm{monic} \\ \mathrm{irreducible} \\ \mathrm{deg}(P)=2g+1}}\sum_{n=0}^{g}\sum_{\substack{f_{1} \ \mathrm{monic} \\ \mathrm{deg}(f_{1})=n \\ f_{1}\neq\square}}\chi_{P}(f_{1})q^{-\tfrac{n}{2}}=O\left(\frac{q^{\tfrac{3}{2}g}}{\log_{q}|P|}g\right).\nonumber
\end{equation}
\end{prop}
\begin{proof}
Let $f_{1}\in\mathbb{F}_{q}[T]$ be a fixed monic nonsquare polynomial such that $\mathrm{deg}(f_{1})<\mathrm{deg}(P)=2g+1$. By the quadratic reciprocity law, Theorem \ref{reciprocity}, we have
\begin{equation}
\left(\frac{P}{f_{1}}\right)=(-1)^{\tfrac{q-1}{2}(2g+1)(\mathrm{deg}(f_{1}))}\left(\frac{f_{1}}{P}\right).
\end{equation}
Note that the sign $(-1)^{\tfrac{q-1}{2}(2g+1)(\mathrm{deg}(f_{1}))}$ is the same for all monic irreducible polynomials $P$ of degree $2g+1$, so
\begin{equation}
\left|\sum_{\substack{P \ \mathrm{monic} \\ \mathrm{irreducible} \\ \mathrm{deg}(P)=2g+1}}\left(\frac{P}{f_{1}}\right)\right|=\left|\sum_{\substack{P \ \mathrm{monic} \\ \mathrm{irreducible} \\ \mathrm{deg}(P)=2g+1}}\left(\frac{f_{1}}{P}\right)\right|.
\end{equation}
Thus we can write
\begin{equation}
\sum_{\substack{P \ \mathrm{monic} \\ \mathrm{irreducible} \\ \mathrm{deg}(P)=2g+1}}\sum_{n=0}^{g}\sum_{\substack{f_{1} \ \mathrm{monic} \\ \mathrm{deg}(f_{1})=n \\ f_{1}\neq\square}}\chi_{P}(f_{1})q^{-\tfrac{n}{2}}\ll\sum_{n=0}^{g}\sum_{\substack{f_{1} \ \mathrm{monic} \\ \mathrm{deg}(f_{1})=n \\ f_{1}\neq\square}}q^{-\tfrac{n}{2}}\left|\sum_{\substack{P \ \mathrm{monic} \\ \mathrm{irreducible} \\ \mathrm{deg}(P)=2g+1}}\left(\frac{f_{1}}{P}\right)\right|\nonumber
\end{equation}
and using the bound for character sums over prime polynomials given in Proposition \ref{bound} we have,
\begin{eqnarray}
\sum_{\substack{P \ \mathrm{monic} \\ \mathrm{irreducible} \\ \mathrm{deg}(P)=2g+1}}\sum_{n=0}^{g}\sum_{\substack{f_{1} \ \mathrm{monic} \\ \mathrm{deg}(f_{1})=n \\ f_{1}\neq\square}}\chi_{P}(f_{1})q^{-\tfrac{n}{2}}&\ll&\sum_{n=0}^{g}q^{-\tfrac{n}{2}}\sum_{\substack{f_{1} \ \mathrm{monic} \\ \mathrm{deg}(f_{1})=n}}n\frac{q^{g}}{2g+1}\nonumber\\
&\ll& \frac{\sqrt{|P|}}{\log_{q}|P|}gq^{\tfrac{g}{2}},\nonumber
\end{eqnarray}
which proves the proposition.
\end{proof}

We can prove a corresponding estimate for the dual sum in \eqref{eq3.1} using the same approach. In the end we have

\begin{prop}
\label{prop3.4}
\begin{equation}
\sum_{\substack{P \ \mathrm{monic} \\ \mathrm{irreducible} \\ \mathrm{deg}(P)=2g+1}}\sum_{m=0}^{g-1}\sum_{\substack{f_{2} \ \mathrm{monic} \\ \mathrm{deg}(f_{2})=m \\ f_{2}\neq\square}}\chi_{P}(f_{2})q^{-\tfrac{m}{2}}=O\left(\frac{q^{\tfrac{3}{2}g}}{\log_{q}|P|}g\right).\nonumber
\end{equation}
\end{prop}

\subsection{Proof of the Theorem for the First Moment}

We are now in a position to prove Theorem \ref{thm1}.

\begin{proof}[Proof of Theorem \ref{thm1}]
We can write
\begin{multline}
\sum_{\substack{P \ \mathrm{monic} \\ \mathrm{irreducible} \\ \mathrm{deg}(P)=2g+1}}(\log_{q}|P|)L(\tfrac{1}{2},\chi_{P})\nonumber\\
=\sum_{\substack{P \ \mathrm{monic} \\ \mathrm{irreducible} \\ \mathrm{deg}(P)=2g+1}}(\log_{q}|P|)\left(\sum_{n=0}^{g}\sum_{\substack{f_{1} \ \mathrm{monic} \\ \mathrm{deg}(f_{1})=n}}\chi_{P}(f_{1})q^{-\tfrac{n}{2}}+\sum_{m=0}^{g-1}\sum_{\substack{f_{2} \ \mathrm{monic} \\ \mathrm{deg}(f_{2})=m}}\chi_{P}(f_{2})q^{-\tfrac{m}{2}}\right)
\end{multline} 
Making use of Propositions \ref{prop3.1}, \ref{prop3.2}, \ref{prop3.3} and \ref{prop3.4} we establish that
\begin{multline}
\sum_{\substack{P \ \mathrm{monic} \\ \mathrm{irreducible} \\ \mathrm{deg}(P)=2g+1}}(\log_{q}|P|)L(\tfrac{1}{2},\chi_{P})=|P|\left(\left[\frac{g}{2}\right]+\left[\frac{g-1}{2}\right]+2\right)+O(q^{\tfrac{3g}{2}}g).\nonumber
\end{multline}
and using that
\begin{equation}
\left[\frac{g}{2}\right]+\left[\frac{g-1}{2}\right]=g-1
\end{equation}
and
\begin{equation}
g+1=\frac{\log_{q}|P|}{2}+\frac{1}{2}
\end{equation}
we conclude the proof of the theorem.
\end{proof}

\section{The Second Moment}

In this section we prove the Theorem \ref{thm2}.

\subsection{Secondary Lemmas}

We will need some auxiliary lemmas before we proceed to the proof of Theorem \ref{thm2}.

The starting point is a representation for $L(\tfrac{1}{2},\chi_{P})^{2}$ which can be viewed as the analogue of the approximate functional equation for a quadratic Dirichlet $L$-function (Lemma 3 in \cite{J}). In this case there is no error term and the formula is exact.

\begin{lem}
\label{funcional2}
Let $\chi_{P}$ be the quadratic Dirichlet character associated to the monic irreducible polynomial $P\in A$. Then
\begin{equation}
\label{eq4.1}
L(\tfrac{1}{2},\chi_{P})^{2}=\sum_{\substack{f_{1} \ \mathrm{monic} \\ \mathrm{deg}(f_{1})\leq2g}}\frac{\chi_{P}(f_{1})d(f_{1})}{|f_{1}|^{\tfrac{1}{2}}}+\sum_{\substack{f_{2} \ \mathrm{monic} \\ \mathrm{deg}(f_{2})\leq2g-1}}\frac{\chi_{P}(f_{2})d(f_{2})}{|f_{2}|^{\tfrac{1}{2}}},
\end{equation}
where $d(f)$ is the divisor function for polynomials $f\in A$ (see \cite[pg.15]{Ro}).
\end{lem}

\begin{proof}
We have $L(s,\chi_{P})=L_{C_{P}}(u)$. So
\begin{equation}
L_{C_{P}}(u)^{2}=((qu^{2})^{g})^{2}L_{C_{P}}\left(\frac{1}{qu}\right)^{2}.
\end{equation}
Writing $L_{C_{P}}(u)^{2}=\sum_{n=0}^{4g}a_{n}u^{n}$ we obtain
\begin{eqnarray}
\sum_{n=0}^{4g}a_{n}u^{n}&=&(qu^{2})^{g}(qu^{2})^{g}\sum_{m=0}^{4g}a_{m}q^{-m}u^{-m}\nonumber\\
&=&\sum_{m=0}^{4g}a_{m}q^{2g-m}u^{4g-m}\nonumber\\
&=&\sum_{k=0}^{4g}a_{4g-k}q^{k-2g}u^{k}.
\end{eqnarray}
Equating coefficients we have that $a_{n}=a_{4g-n}q^{n-2g}$ and so we can write
\begin{equation}
\label{func2}
L_{C_{P}}(u)^{2}=\sum_{n=0}^{2g}a_{n}u^{n}+((qu^{2})^{g})^{2}\sum_{m=0}^{2g-1}a_{m}q^{-m}u^{-m}.
\end{equation}
From $L(s,\chi_{P})^{2}$ we see that the coefficients $a_{n}$ are given by
\begin{equation}
a_{n}=\sum_{\substack{f \ \mathrm{monic} \\ \mathrm{deg}(f)=n}}\chi_{P}(f)d(f),
\end{equation}
where 
\begin{equation}
d(f)=\sum_{\substack{h_{1}h_{2}=f \\ h_{1},h_{2} \ \mathrm{monic}}}1.
\end{equation}
Therefore writing $s=1/2$, i.e. $u=q^{-1/2}$, in \eqref{func2} proves the lemma.
\end{proof}

Our next lemma is quoted from Rosen \cite[Proposition 2.5]{Ro}

\begin{lem}
\label{divisorbound}
\begin{equation}
\sum_{\substack{f \ \mathrm{monic} \\ \mathrm{deg}(f)=n}}d(f)\ll q^{n}n.
\end{equation}
\end{lem}

The next lemma is a minor modification of Theorem 17.4 in \cite{Ro}

\begin{lem}
\label{lem4.3}
Let $f:A^{+}\rightarrow\mathbb{C}$ and let $\zeta_{f}(s)$ be the corresponding Dirichlet series. Suppose this series converges absolutely in the region $\mathfrak{R}(s)>1$ and is holomorphic in the region $\{s\in B:\mathfrak{R}(s)=1\}$ except for a simple pole of order $r$ at $s=1$, where $A^{+}$ denotes the set of monic polynomials in $\mathbb{F}_{q}[T]$ and 
$$B=\left\{s\in\mathbb{C}:-\frac{\pi i}{\log(q)}\leq\mathcal{I}(s)\leq\frac{\pi i}{\log(q)}\right\}.$$
Let $\alpha=\lim_{s\rightarrow1}(s-1)^{r}\zeta_{f}(s)$. Then, there is a $\delta<1$ and constants $c_{-i}$ with $1\leq i\leq r$ such that
\begin{equation}
\sum_{\substack{\mathrm{deg}(D)=n}}f(D)=q^{n}\left(\sum_{i=1}^{r}c_{-i}\binom{n+i-1}{i-1}(-q)^{i}\right)+O(q^{\delta n}).
\end{equation}
And the sum in parenthesis is a polynomial in $n$ of degree $r-1$ with leading term
$$\frac{\log(q)^{r}}{(r-1)!}\alpha n^{r-1}.$$
\end{lem}

\begin{lem}
Let $f$ be a monic polynomial in $A=\mathbb{F}_{q}[T]$. Then
\begin{equation}
\sum_{\substack{f \ \mathrm{monic} \\ \mathrm{deg}(f)=n}}d(f^{2})=\frac{1}{2}\frac{1}{\zeta_{A}(2)}q^{n}n^{2}+O(q^{n}n).
\end{equation}
\end{lem}
\begin{proof}
We consider the Dirichlet series associated to $d(f^{2})$
\begin{eqnarray}
\zeta_{f}(s)=\sum_{f \ \mathrm{monic}}\frac{d(f^{2})}{|f|^{s}}&=&\prod_{\substack{P \ \mathrm{monic} \\ \mathrm{irreducible}}}\left(1+\frac{d(P^{2})}{|P|^{s}}+\frac{d(P^{4})}{|P|^{2s}}+\cdots\right)\nonumber\\
&=&\prod_{\substack{P \ \mathrm{monic} \\ \mathrm{irreducible}}}\left(1+\left(\frac{-3}{|P|^{s}(|P|^{s}-1)^{2}}+\frac{1}{(|P|^{s}-1)^{2}}\right)\right)\nonumber\\
&=&\frac{\zeta_{A}(s)^{3}}{\zeta_{A}(2s)}.\nonumber
\end{eqnarray}
From \eqref{eq:zetaA} the sum converges absolutely for $\mathfrak{R}(s)>1$, is holomorphic on the disc $\{u=q^{-s}\in\mathbb{C}:|u|\leq q^{-\delta}\}$ for some $\delta<1$, and $\zeta_{f}(s)$ has a pole of order $3$ at $s=1$. We now apply Lemma \ref{lem4.3} to obtain
\begin{equation}
\sum_{\substack{f \ \mathrm{monic} \\ \mathrm{deg}(f)=n}}d(f^{2})=\frac{(\log q)^{3}}{2}\alpha q^{n}n^{2}+O(q^{n}n),
\end{equation}
where
\begin{equation}
\alpha=\lim_{s\rightarrow1}(s-1)^{3}\frac{\zeta_{A}(s)^{3}}{\zeta_{A}(2s)}=\frac{q-1}{q(\log q)^{3}}.
\end{equation}
\end{proof}

\subsection{Preparation for the Proof}

From Lemma \ref{funcional2}, $L(\tfrac{1}{2},\chi_{P})^{2}$ can be written as two similar sums. Our main aim in this section is to average, over the prime polynomials, the first sum in the right-hand side of \eqref{eq4.1}. We start by writing
\begin{multline}
\sum_{\substack{f_{1} \ \mathrm{monic} \\ \mathrm{deg}(f_{1})\leq2g}}\frac{\chi_{P}(f_{1})d(f_{1})}{|f_{1}|^{\tfrac{1}{2}}}\\
=\sum_{\substack{f_{1} \ \mathrm{monic} \\ \mathrm{deg}(f_{1})\leq2g \\ f_{1}=\square}}\frac{\chi_{P}(f_{1})d(f_{1})}{|f_{1}|^{\tfrac{1}{2}}}+\sum_{\substack{f_{1} \ \mathrm{monic} \\ \mathrm{deg}(f_{1})\leq2g \\ f_{1}\neq\square}}\frac{\chi_{P}(f_{1})d(f_{1})}{|f_{1}|^{\tfrac{1}{2}}}.
\end{multline}

\subsection{The Main Term} The following proposition is established in this section.

\begin{prop}
\label{prop4.5}
\begin{multline}
\sum_{\substack{P \ \mathrm{monic} \\ \mathrm{irrducible} \\ \mathrm{deg}(P)=2g+1}}\sum_{\substack{f_{1} \ \mathrm{monic} \\ \mathrm{deg}(f_{1})\leq2g \\ f_{1}=\square}}\frac{\chi_{P}(f_{1})d(f_{1})}{|f_{1}|^{\tfrac{1}{2}}}\\
=\frac{1}{12}\frac{1}{\zeta_{A}(2)}\frac{|P|}{\log_{q}|P|}g(g+1)(2g+1)+O\left(\frac{|P|}{\log_{q}|P|}g^{2}\right).\nonumber
\end{multline}
\end{prop}
\begin{proof}
We have
\begin{multline}
\sum_{\substack{P \ \mathrm{monic} \\ \mathrm{irrducible} \\ \mathrm{deg}(P)=2g+1}}\sum_{\substack{f_{1} \ \mathrm{monic} \\ \mathrm{deg}(f_{1})\leq2g \\ f_{1}=\square}}\frac{\chi_{P}(f_{1})d(f_{1})}{|f_{1}|^{\tfrac{1}{2}}}\nonumber\\
\ \ \ \ \ \ \ \ \ \ \ \ \ \ \ \ \ \ \ \ \ \ \ \ \ \ \ \ \ \ \ \ \ \ \ \ \ \ \ \ \ \ \ =\sum_{n=0}^{2g}q^{-\tfrac{n}{2}}\sum_{\substack{f_{1}=\square \\ \mathrm{deg}(f_{1})=n}}d(f_{1})\sum_{\substack{P \ \mathrm{monic} \\ \mathrm{irrducible} \\ \mathrm{deg}(P)=2g+1}}\chi_{P}(f_{1})\ \ \ \ \ \ \ \ \ \ \ \ \ \ \ \ \ \ \ \ \ \ \ \ \ \ \ \ \ \ \ \ \ \ \ \ \ \ \ \ \ \ \ \ \ \ \ \ \ \ \ \ \\
\ \ \ \ \ \ \ \ \ \ \ \ \ \ \ \ \ \ \ \ \ \ \ \ \ \ \ \ \ \ \ \ \ \ \ \ \ \ \ \ \ \ \ =\sum_{m=0}^{g}q^{-m}\sum_{\substack{l \ \mathrm{monic} \\ \mathrm{deg}(l)=m}}d(l^{2})\sum_{\substack{P \ \mathrm{monic} \\ \mathrm{irrducible} \\ \mathrm{deg}(P)=2g+1}}1.\ \ \ \ \ \ \ \ \ \ \ \ \ \ \ \ \ \ \ \ \ \ \ \ \ \ \ \ \ \ \ \ \ \ \ \ \ \ \ \ \ \ \ \ \ \ \ \ \ \ \ \ 
\end{multline}
We again make use of the Prime Polynomial Theorem \ref{thm:pnt} to obtain
\begin{multline}
\sum_{\substack{P \ \mathrm{monic} \\ \mathrm{irrducible} \\ \mathrm{deg}(P)=2g+1}}\sum_{\substack{f_{1} \ \mathrm{monic} \\ \mathrm{deg}(f_{1})\leq2g \\ f_{1}=\square}}\frac{\chi_{P}(f_{1})d(f_{1})}{|f_{1}|^{\tfrac{1}{2}}}\nonumber\\
=\frac{|P|}{\log_{q}|P|}\sum_{m=0}^{g}q^{-m}\sum_{\substack{l \ \mathrm{monic} \\ \mathrm{deg}(l)=m}}d(l^{2})+O\left(\frac{\sqrt{|P|}}{\log_{q}|P|}\sum_{m=0}^{g}q^{-m}\sum_{\substack{l \ \mathrm{monic} \\ \mathrm{deg}(l)=m}}d(l^{2})\right)
\end{multline}
Invoking lemma \ref{lem4.3} we obtain the following equation
\begin{multline}
\sum_{\substack{P \ \mathrm{monic} \\ \mathrm{irrducible} \\ \mathrm{deg}(P)=2g+1}}\sum_{\substack{f_{1} \ \mathrm{monic} \\ \mathrm{deg}(f_{1})\leq2g \\ f_{1}=\square}}\frac{\chi_{P}(f_{1})d(f_{1})}{|f_{1}|^{\tfrac{1}{2}}}\nonumber\\
=\frac{|P|}{\log_{q}|P|}\frac{1}{2}\frac{1}{\zeta_{A}(2)}\sum_{m=0}^{g}m^{2}+O\left(\frac{|P|}{\log_{q}|P|}\sum_{m=0}^{g}m\right)+O\left(\frac{\sqrt{|P|}}{\log_{q}|P|}\sum_{m=0}^{g}m^{2}\right)\\ \\
=\frac{|P|}{\log_{q}|P|}\frac{1}{12}\frac{1}{\zeta_{A}(2)}g(g+1)(2g+1)+O\left(\frac{|P|}{\log_{q}|P|}g^{2}\right).
\end{multline}
\end{proof}

In a similar way we can prove that

\begin{prop}
\label{prop4.6}
\begin{multline}
\sum_{\substack{P \ \mathrm{monic} \\ \mathrm{irrducible} \\ \mathrm{deg}(P)=2g+1}}\sum_{\substack{f_{2} \ \mathrm{monic} \\ \mathrm{deg}(f_{2})\leq2g-1 \\ f_{2}=\square}}\frac{\chi_{P}(f_{2})d(f_{2})}{|f_{2}|^{\tfrac{1}{2}}}\nonumber\\ \\
=\frac{1}{12}\frac{1}{\zeta_{A}(2)}\frac{|P|}{\log_{q}|P|}\left[\frac{2g-1}{2}\right]\left(1+\left[\frac{2g-1}{2}\right]\right)\left(1+2\left[\frac{2g-1}{2}\right]\right)\\ \\
+O\left(\frac{|P|}{\log_{q}|P|}g^{2}\right).\ \ \ \ \ \ \ \ \ \ \ \ \ \ \ \ \ \ \ \ \ \ \ \ \ \ \ \ \ \ \ \ \ \ \ \ \ \ \ \ \ \ \ \ \ \ \ \ \ \ \ \ \ \ 
\end{multline}
\end{prop}

\subsection{Contributions of non--squares}

The main result in this section is given by the following proposition.

\begin{prop}
\label{prop4.7}
We have that,
\begin{equation}
\sum_{\substack{P \ \mathrm{monic} \\ \mathrm{irrducible} \\ \mathrm{deg}(P)=2g+1}}\sum_{\substack{f_{1} \ \mathrm{monic} \\ \mathrm{deg}(f_{1})\leq2g \\ f_{1}\neq\square}}\frac{\chi_{P}(f_{1})d(f_{1})}{|f_{1}|^{\tfrac{1}{2}}}=O\left(|P|g\right).
\end{equation}
\end{prop}

\begin{proof}
\begin{eqnarray}
\sum_{\substack{P \ \mathrm{monic} \\ \mathrm{irrducible} \\ \mathrm{deg}(P)=2g+1}}\sum_{\substack{f_{1} \ \mathrm{monic} \\ \mathrm{deg}(f_{1})\leq2g \\ f_{1}\neq\square}}\frac{\chi_{P}(f_{1})d(f_{1})}{|f_{1}|^{\tfrac{1}{2}}}&\ll&\sum_{\substack{f_{1} \ \mathrm{monic} \\ \mathrm{deg}(f_{1})\leq2g \\ f_{1}\neq\square}}\frac{d(f_{1})}{|f_{1}|^{\tfrac{1}{2}}}\left|\sum_{\substack{P \ \mathrm{monic} \\ \mathrm{irrducible} \\ \mathrm{deg}(P)=2g+1}}\left(\frac{f_{1}}{P}\right)\right|\nonumber\\
\label{4.14}&\ll&\frac{\sqrt{|P|}}{2g+1}\sum_{n=0}^{2g}\frac{n}{q^{n/2}}\sum_{\substack{f_{1} \ \mathrm{monic} \\ \mathrm{deg}(f_{1})=n}}d(f_{1})\\
&\ll&\frac{\sqrt{|P|}}{2g+1}\sum_{n=0}^{2g}n^{2}q^{n/2}\nonumber\\
&\ll&|P|g,\nonumber
\end{eqnarray}
where we have used Proposition \ref{bound} in the first line and Lemma \ref{divisorbound} in \eqref{4.14}.
\end{proof}

Similarly we have

\begin{prop}
\label{prop4.8}
\begin{equation}
\sum_{\substack{P \ \mathrm{monic} \\ \mathrm{irrducible} \\ \mathrm{deg}(P)=2g+1}}\sum_{\substack{f_{2} \ \mathrm{monic} \\ \mathrm{deg}(f_{2})\leq2g-1 \\ f_{2}\neq\square}}\frac{\chi_{P}(f_{2})d(f_{2})}{|f_{2}|^{\tfrac{1}{2}}}=O\left(|P|g\right).
\end{equation}
\end{prop}

\subsection{Proof of Theorem for the Second Moment}
We are now in a position to prove Theorem \ref{thm2}.

\begin{proof}[Proof of Theorem \ref{thm2}]
We can write
\begin{multline}
\sum_{\substack{P \ \mathrm{monic} \\ \mathrm{irreducible} \\ \mathrm{deg}(P)=2g+1}}L(\tfrac{1}{2},\chi_{P})^{2}\nonumber\\
=\sum_{\substack{P \ \mathrm{monic} \\ \mathrm{irreducible} \\ \mathrm{deg}(P)=2g+1}}\left(\sum_{\substack{f_{1} \ \mathrm{monic} \\ \mathrm{deg}(f_{1})\leq2g}}\frac{\chi_{P}(f_{1})d(f_{1})}{|f_{1}|^{\tfrac{1}{2}}}+\sum_{\substack{f_{2} \ \mathrm{monic} \\ \mathrm{deg}(f_{2})\leq2g-1}}\frac{\chi_{P}(f_{2})d(f_{2})}{|f_{2}|^{\tfrac{1}{2}}}\right)
\end{multline} 
Making use of Propositions \ref{prop4.5}, \ref{prop4.6}, \ref{prop4.7} and \ref{prop4.8} we establish that
\begin{multline}
\sum_{\substack{P \ \mathrm{monic} \\ \mathrm{irreducible} \\ \mathrm{deg}(P)=2g+1}}L(\tfrac{1}{2},\chi_{P})^{2}=\frac{1}{12}\frac{1}{\zeta_{A}(2)}\frac{|P|}{\log_{q}|P|}\nonumber\\
\times\Bigg[g(g+1)(2g+1)+\left[\frac{2g-1}{2}\right]\left(1+\left[\frac{2g-1}{2}\right]\right)\left(1+2\left[\frac{2g-1}{2}\right]\right)\Bigg]\nonumber\\
+O\left(|P|g\right).\ \ \ \ \ \ \ \ \ \ \ \ \ \ \ \ \ \ \ \ \ \ \ \ \ \ \ \ \ \ \ \ \ \ \ \ \ \ \ \ \ \ \ \ \ \ \ \ \ \ \ \ \ \ \ \ \ \ \ \ \ \ \ \ \ \ \ \ \ \ \ \ \ \ \ 
\end{multline}
We use that
\begin{equation}
\left[\frac{2g-1}{2}\right]\left(1+\left[\frac{2g-1}{2}\right]\right)\left(1+2\left[\frac{2g-1}{2}\right]\right)=(g-1)g(2g-1)
\end{equation}
and
\begin{equation}
g(g+1)(2g+1)+(g-1)g(2g-1)=4g^{3}+O(g)
\end{equation}
and after some simple arithmetical manipulations this gives the desired formula.
\end{proof}

\section{Acknowledgments}
We would like to thank Professor Ze\'{e}v Rudnick for suggesting the problems tackled in this paper and the Professors Michael Rosen and Jeffrey Hoffstein for helpful and interesting discussions.


\end{document}